\documentclass[11pt]{amsart}
\usepackage{amsxtra}
\usepackage{amssymb}
\usepackage{enumerate}
\theoremstyle{plain}

\newcommand{\cleqn}{\setcounter{equation}{0}}
\newcommand{\clth}{\setcounter{theorem}{0}}
\newcommand {\sectionnew}[1]{\section{#1}\cleqn\clth}

\newtheorem{theorem}{Theorem}[section]
\newtheorem{lemma}[theorem]{Lemma}
\newtheorem{definition-theorem}[theorem]{Definition-Theorem}
\newtheorem{proposition}[theorem]{Proposition}
\newtheorem{corollary}[theorem]{Corollary}
\newtheorem{example}[theorem]{Example}
\newtheorem{conjecture}[theorem]{Conjecture}
\theoremstyle{definition}
\newtheorem{remark}[theorem]{Remark}
\newtheorem{definition}[theorem]{Definition}
\newtheorem*{remark*}{Remark}
\newtheorem{question}[theorem]{Question}
\newtheorem{notation}[theorem]{Notation}
\newcommand \bth[1] { \begin{theorem}\label{t#1} }
\newcommand \ble[1] { \begin{lemma}\label{l#1} }

\newcommand \bpr[1] { \begin{proposition}\label{p#1} }
\newcommand \bco[1] { \begin{corollary}\label{c#1} }
\newcommand \bde[1] { \begin{definition}\label{d#1}\rm }
\newcommand \bex[1] { \begin{example}\label{e#1}\rm }
\newcommand \bre[1] { \begin{remark}\label{r#1}\rm }
\newcommand \bcj[1] { \begin{conjecture}\label{j#1}\rm }

\newcommand \bnota[1] { \begin{notation}\label{n#1}\rm }
\renewcommand {\eth} { \end{theorem} }
\newcommand {\ele} { \end{lemma} }

\newcommand {\epr} { \end{proposition} }
\newcommand {\eco} { \end{corollary} }
\newcommand {\ede} { \end{definition} }
\newcommand {\eex} { \end{example} }
\newcommand {\ere} { \end{remark} }
\newcommand {\ecj} { \end{conjecture} }

\newcommand {\enota} { \end{notation} }

\newcommand \Aset {{\mathbb A}}        
\newcommand \KK {{\mathbb K}}
\newcommand \Zset {{\mathbb Z}}

\newcommand \OO {{\mathcal{O}}}

\newcommand \de {\delta}
\newcommand \al {\alpha}

\newcommand \la {\lambda}

\newcommand \sig {\sigma}

\newcommand \id { {\mathrm{id}} }

\DeclareMathOperator \gr  { {\mathrm{gr}} }

\DeclareMathOperator \Fract { {\mathrm{Fract}} }

\newcommand\kx{\KK^*}

\newcommand\HH{{\mathcal{H}}}

\DeclareMathOperator \Spec {Spec}

\newcommand \Znn {\Zset_{\ge 0}}

\newcommand \Rhat {\widehat{R}}

\newcommand\Hspec{\HH{\text{-}}\Spec}
\newcommand{\gc}{ [ \hspace{-0.65mm} [}
\newcommand{\dc}{]  \hspace{-0.65mm} ]}
\newcommand \bfq {\mathbf q}
\newcommand \bfla {\boldsymbol \lambda}
\DeclareMathOperator\height{ht}
\DeclareMathOperator\GK{GKdim}

\begin{document}

\title{Catenarity in quantum nilpotent algebras}

\author[K. R. Goodearl]{K. R. Goodearl}
\address{
Department of Mathematics \\
University of California\\
Santa Barbara, CA 93106 \\
U.S.A.
}
\email{goodearl@math.ucsb.edu}

\author[S. Launois]{S. Launois}
\address{
School of Mathematics, Statistics and Actuarial Science \\
University of Kent \\
Canterbury, Kent, CT2 7FS \\
 UK
}
\email{S.Launois@kent.ac.uk}

\begin{abstract}
In this paper, it is established that quantum nilpotent algebras (also known as CGL extensions) are catenary, i.e., all saturated chains of inclusions of prime ideals between any two given prime ideals $P \subsetneq Q$ have the same length. This is achieved by proving that the prime spectra of these algebras have normal separation, and then establishing the mild homological conditions necessary to apply a result of Lenagan and the first author. The work also recovers the Tauvel height formula for quantum nilpotent algebras, a result that was first obtained by Lenagan and the authors through a different approach.
\end{abstract}



\thanks{The research of the first named author was supported
by US National Science Foundation grant DMS-1601184. That of the second named author was supported by EPSRC grant EP/N034449/1.}

\maketitle


\sectionnew{Introduction}

The aim of this paper is to study the prime spectra of quantum algebras. More precisely, we focus on the catenary property -- that all saturated chains of inclusions of prime ideals between any two fixed prime ideals have the same length -- for a large class of (quantum) algebras called quantum nilpotent algebras. Examples of these algebras include for instance quantum matrices and more generally quantum Schubert cells. Quantum nilpotent algebras have also appeared in the literature under the name ``CGL extensions", and their prime spectra have been proved in some cases to be linked to totally nonnegative matrix varieties; see for instance \cite{gll1,gll2,LLN} for more details.
\medskip

A fundamental property of any affine algebraic variety $V$ is that all saturated chains of inclusions of irreducible subvarieties of $V$ between any two fixed irreducible subvarieties have the same length.  Restated in terms of the coordinate ring $\OO(V)$, this says that the prime spectrum of $\OO(V)$ is catenary.

Quantized coordinate rings of affine varieties are expected to enjoy suitable versions of the properties of their classical counterparts. In particular, it is conjectured that the prime spectra of quantized coordinate rings must be catenary.  This conjecture has been verified for the quantized coordinate rings of many varieties, such as matrix varieties \cite{Ca2}, affine spaces and general and special linear groups \cite{GLen}, simple algebraic groups \cite{GZ, Yak3}, Schubert cells \cite{Yak1}, and Grassmannians \cite{LLR2}. Catenarity has also been established for many related quantum algebras, such as uni- and multiparameter quantum symplectic and euclidean spaces \cite{Oh1, Hrt}, quantized Weyl algebras \cite{GLen, Oh1}, and twisted quantum Schubert cell algebras \cite{Yak2}. The above references deal with generic quantum algebras, those whose quantum parameters are non-roots of unity. When the quantum parameters are roots of unity, such algebras satisfy polynomial identities, and catenarity of affine polynomial identity algebras follows from a result of Schelter \cite[Theorem 1]{Sch}.

Here we establish catenarity for all members of the broad family of quantum nilpotent algebras (defined below). These algebras and localizations thereof cover the generic quantum algebras mentioned above, except for quantized coordinate rings of simple algebraic groups and Grassmannians. 
\medskip

By a famous result of Gabber, enveloping algebras of finite dimensional solvable Lie algebras are catenary (see, e.g., \cite{Gab} or a combination of \cite[Appendix Al]{LevSt} and \cite[Ch. 9]{KrLen}). This result was extended to enveloping algebras of finite dimensional solvable Lie superalgebras by Lenagan \cite{Len}. The method of proof involved establishing good homological properties of the ring, connecting homological properties with growth, and controlling growth properties of prime factors by finding normal elements. (A \emph{normal element} in a ring $R$ is an element $x$ such that $xR=Rx$.) Abstracting these methods, Lenagan and the first author gave a set of homological and ring-theoretical conditions that ensure catenarity of an algebra 
\cite[Theorem 7.1]{GLen}. The method additionally yields the following useful height formula, first established by Tauvel \cite{tauvel} for enveloping algebras of solvable Lie algebras:
$$
\height(P) + \GK(R/P) = \GK(R) \qquad \text{for all prime ideals}\; P\; \text{of}\; R.
$$
This formula has been proved for many quantum algebras such as the ones mentioned above, and Lenagan and the present authors recently proved that all quantum nilpotent algebras satisfy Tauvel's height formula \cite{GLL}.

In order to apply the above methods to an algebra $R$, a suitable supply of normal elements in prime factor algebras is needed, in the following form. The prime spectrum $\Spec R$ must have \emph{normal separation}, meaning that for any pair of distinct comparable prime ideals $P \subsetneq Q$ in $R$, the factor $Q/P$ contains a nonzero normal element of $R/P$. Normal separation was proved by Cauchon for quantum matrices \cite{Ca2} using ring-theoretical and combinatorial methods. Later, Yakimov established it for quantum Schubert cells \cite{Yak1} using representation theoretical methods. Here we prove it for a larger class of algebras using purely ring-theoretical methods.

Establishing normal separation for quantum nilpotent algebras requires most of the effort in the paper, since existing results can be applied to verify the required homological properties.

\subsection{Quantum nilpotent algebras}
Let  $R$ an iterated skew polynomial algebra of the form
\begin{equation}  \label{itOre}
R\ = \ \KK[x_1][x_2;\sigma_2,\delta_2]\cdots[x_N;\sigma_N,\delta_N],
\end{equation}
over a field $\KK$,
 where $\sigma_j$ is an automorphism of the $\KK$-algebra 
 $$
R_{j-1} :=\KK[x_1][x_2;\sigma_2,\delta_2]\dots[x_{j-1};\sigma_{j-1},\delta_{j-1}]
$$
 and  $\delta_j$ is a $\KK$-linear $\sigma_j$-derivation of
 $R_{j-1}$, for all $j\in \gc 2 ,N \dc$. (When needed, we denote $R_0 := \KK$ and set $R_1 = \KK[x_1; \sigma_1, \delta_1]$ with $\sigma_1 := \id_\KK$, $\delta_1 := 0$.) In particular, $R$ and the $R_j$ are noetherian domains. 
 
 \begin{definition}  \label{CGLdef}
An iterated skew polynomial extension $R$ as in \eqref{itOre} is called a \emph{quantum nilpotent algebra} or a \emph{CGL extension}
\cite[Definition 3.1]{LLR} if it is 
equipped with a rational action of a $\KK$-torus $\HH = (\kx)^d$ 
by $\KK$-algebra automorphisms satisfying the following conditions:
\begin{enumerate}[(i)]
\item  The elements $x_1, \ldots, x_N$ are $\HH$-eigenvectors.
\item For every $j \in \gc 2,N \dc$, $\delta_j$ is a locally nilpotent 
$\sigma_j$-derivation of $R_{j-1}$. 
\item For every $j \in \gc 1,N \dc$, there exists $h_j \in \HH$ such that $(h_j\cdot)|_{R_{j-1}} = \sigma_j$ and
$h_j \cdot x_j = q_j x_j$ for some $q_j \in \kx$ which is 
not a root of unity. 
\end{enumerate}
(We have omitted the condition $\sigma_j  \delta_j = q_j \delta_j \sigma_j$ from the original definition, as it follows from the other conditions; see, e.g., \cite[Eq. (3.1); comments, p.694]{GYncufd}.) 
\end{definition}

The main theorem of the paper is

\begin{theorem}  \label{catCGL}
If $R$ is a quantum nilpotent algebra, then $\Spec R$ is catenary, and all prime quotients of $R$ satisfy Tauvel's height formula.
\end{theorem}

The key requirement in proving this theorem is normal separation in $\Spec R$. Existence of suitable normal elements is established by induction on the number of indeterminates in $R$. The following two sections are devoted to the induction step, in which normal elements are constructed in certain skew polynomial algebras in one indeterminate and factor algebras thereof. Normal separation for quantum nilpotent algebras is achieved in Section \ref{proofmaintheorem} together with the desired homological properties, and Theorem \ref{catCGL} is proved there.

\subsection{Notation and conventions} Throughout, all algebras will be unital algebras over a fixed base field $\KK$. All the skew polynomial rings we consider will be of the form $A[X;\sig,\de]$ where the coefficient ring $A$ is a $\KK$-algebra, $\sig$ is a $\KK$-algebra automorphism of $A$, and $\de$ is a $\KK$-linear left $\sig$-derivation of $A$. The $\KK$-automorphism and $\KK$-linearity assumptions ensure that $A[X;\sig,\de]$ is a $\KK$-algebra, and that it is noetherian if $A$ is noetherian. The indeterminate $X$ in $A[X;\sig,\de]$ skew-commutes with elements $a\in A$ as follows: $Xa = \sig(a)X + \de(a)$.

\sectionnew{A first construction of normal elements}

\subsection{Basic assumptions}  \label{basicR}
Let $A$ be a noetherian $\KK$-algebra domain and $R = A[X;\sig,\de]$ a skew polynomial extension.

Assume throughout this section that 
\begin{itemize}
\item $\de$ is locally nilpotent.
\item There is an abelian group $\HH$ acting on $R$ by $\KK$-algebra automorphisms such that $X$ is an $\HH$-eigenvector and $A$ is $\HH$-stable.
\item There exists $h_\circ \in \HH$ such that $(h_\circ \cdot)|_A = \sig$ and the $h_\circ$-eigen\-value $\la_\circ$ of $X$ is not a root of unity.
\end{itemize}
As noted in \cite[Eq.~(3.1)]{GYncufd}, $\sig \de = \la_\circ \de \sig$. More generally \cite[Eq.~(3.2)]{GYncufd},
\begin{equation}  \label{hdelta}
(h\cdot)|_A \circ \de = \chi_X(h) \de \circ (h\cdot)|_A \qquad \forall\, h \in \HH,
\end{equation}
where $\chi_r : \HH \rightarrow \kx$ denotes the $\HH$-eigenvalue of an $\HH$-eigenvector $r\in R$.

\subsection{$\HH$-ideals} Recall that if $C$ is a ring equipped with an action of a group $\HH$ by automorphisms, then the \emph{$\HH$-ideals} of $C$ are the (two-sided) ideals of $C$ invariant under the $\HH$-action. An \emph{$\HH$-prime {\rm(}ideal\/{\rm)}} of $C$ is any proper $\HH$-ideal $P$ such that a product $I_1I_2$ of $\HH$-ideals of $C$ is contained in $P$ only if $I_1$ or $I_2$ is contained in $P$. The ring $C$ is said to be \emph{$\HH$-simple} provided $C\ne 0$ and the only $\HH$-ideals of $C$ are $0$ and $C$. The latter condition is equivalent to the condition that $0$ is the only $\HH$-prime of $C$.

\subsection{Cauchon extensions}  \label{cauchonext}
If in addition to \S\ref{basicR} we assume that
\begin{itemize}
\item Every $\HH$-prime of $A$ is completely prime,
\end{itemize}
then $R$ is a \emph{Cauchon extension} \cite[Definition 2.5]{LLR}.

\subsection{}
Since $\de$ is locally nilpotent, the set $S := \{ X^n \mid n \in \Znn\}$ is a denominator set in $R$ \cite[Lemme 2.1]{Ca1}. Set $ \Rhat := RS^{-1}$. Since the elements of $S$ are $\HH$-eigenvectors, the action of $\HH$ on $R$ extends uniquely to an action by $\KK$-algebra automorphisms on $\Rhat$.

Let $\theta : A \rightarrow \Rhat$ be the \emph{Cauchon map} defined by
\begin{equation}  \label{cauchonmap}
\theta(a) = \sum_{l=0}^\infty \dfrac{(1-\la_\circ)^{-l}}{(l)!_{\la_\circ}}\, \de^l \sig^{-l}(a) X^{-l}.
\end{equation}
(See \eqref{qnum} for the definition of $(l)!_{\la_\circ}$.)
Cauchon established in \cite[Propositions 2.1--2.4]{Ca1} that
\begin{itemize}
\item $\theta$ is an injective $\KK$-algebra homomorphism.
\item $\theta$ extends uniquely to an injective $\KK$-algebra homomorphism $A[Y;\sig] \rightarrow \Rhat$ with $\theta(Y) = X$.
\item Set $B := \theta(A)$ and $T := \theta(A[Y;\sig]) \subseteq \Rhat$. Then $T = B[X;\al]$ where $\al$ is the $\KK$-algebra automorphism of $B$ defined by $\al(\theta(a)) = \theta(\sig(a))$.
\item $S$ is also a denominator set in $T$, and $TS^{-1} = S^{-1}T = \Rhat$.
\end{itemize}
As is noted in \cite[p.327]{LLR}, $B\cap R \subseteq A$. 

By \cite[Lemma 2.6]{LLR} (whose proof only uses the assumptions of \S\ref{basicR}), $\theta$ is $\HH$-equivariant. Since the action of $\sig$ on $A$ is given by $h_\circ$, it follows that $\al = (h_\circ\cdot)|_B$.

\begin{lemma}  \label{theta(a)Xs}
Let $a\in A \setminus \{0\}$ and let $s\in \Znn$ be maximal such that $\de^s(a) \ne 0$. Then $s$ is minimal such that $\theta(a) X^s \in R$.
\end{lemma}

\begin{proof} Since $\de^l(a) = 0$ for $l>s$, we have $\theta(a) = \sum_{l=0}^s c_l \de^l \sig^{-l}(a) X^{-l}$ for some $c_l \in \kx$. Obviously $\theta(a) X^s \in R$.

Suppose that $s>0$ and $\theta(a) X^t \in R$ for some $t<s$. Then $\theta(a) X^{s-1} \in R$, from which it follows that $\de^s \sig^{-s}(a) X^{-1} \in R$. Now $\de^s \sig^{-s}(a) \in A \cap RX$, whence $\de^s \sig^{-s}(a) = 0$. But $\de^s \sig^{-s} = \la_\circ^{s^2} \sig^{-s} \de^s$, so we obtain $\de^s(a) =0$, contradicting our hypotheses. Therefore $s$ is minimal such that $\theta(a) X^s \in R$.
\end{proof}

The following lemma is excerpted from the proof of \cite[Proposition 2.9]{LLR}.

\begin{lemma}  \label{firstnormal}
Let $a\in A$ be a normal $\HH$-eigenvector, and let $s\in \Znn$ be maximal such that $\de^s(a) \ne 0$. Then the element $x := \theta(a) X^s$ is a normal $\HH$-eigenvector in $R$. In particular, $xX = \eta^{-1} Xx$,
where $\eta$ is the $\sig$-eigenvalue of $a$.
\end{lemma}

\begin{proof} Since $\theta$ is $\HH$-equivariant, the element $b := \theta(a)$ is a normal $\HH$-eigenvector in $B$, and the $h_\circ$-eigenvalue of $b$ equals that of $a$, namely $\eta$. By Lemma \ref{theta(a)Xs}, $s$ is minimal such that $b X^s \in R$. This places $x$ in $R$, and clearly $x$ is  an $\HH$-eigenvector. 

Since
\begin{equation}  \label{Xb}
Xb = \al(b)X = h_\circ(b)X = \eta bX,
\end{equation}
we see that $xX = \eta^{-1} Xx$. Moreover, we see that $b$ is also normal in $T$ and in $\Rhat$. In particular, $b\Rhat = \Rhat b$ is an ideal of $\Rhat$. But $b\Rhat = x\Rhat$, and $\Rhat b = \Rhat x$ because $x = \eta^{-s} X^s b$. Thus,
$$
I := x\Rhat \cap R = \Rhat x \cap R
$$
is an ideal of $R$.
We show that $I = Rx = xR$, which will prove that $x$ is normal in $R$. Obviously $I$ contains $Rx$ and $xR$.

Let $y \in I$. Then $y \in b\Rhat$ implies $yX^u \in bT = Tb$ for some $u\ge0$. Now $yX^u = cb$ for some $c\in T$, and $cX^v \in R$ for some $v\ge0$. From \eqref{Xb}, we obtain
$$
y X^{u+v+s} = cb X^{v+s} = \eta^{-v} c X^v b X^s = \eta^{-v} c X^v x \in Rx.
$$
Let $t\in \Znn$ be minimal such that $yX^t \in Rx$, and write $yX^t = rx$ for some $r\in R$.

We wish to show that $t=0$. Write
$$
r = \sum_{i\ge0} r_iX^i \,, \qquad y = \sum_{i\ge0} y_iX^i \,, \qquad x = \sum_{i\ge0} x_iX^i
$$
for some $r_i, y_i, x_i \in A$. In case $s=0$, we would have
$x = b = a \in A$ and so $x_0= a \ne 0$. In case $s>0$, we would have
$$
x_0 X^{-1} + \sum_{i\ge1} x_i X^{i-1} = xX^{-1} = bX^{s-1} \notin R
$$
by the minimality of $s$, so again $x_0 \ne 0$. Thus, $x_0 \ne 0$ in all cases.

Observe that
\begin{align*}
\sum_{i\ge0} y_i X^{i+t} &= yX^t = rx = \sum_{i\ge0} r_i X^i bX^s = \sum_{i\ge0} \eta^i r_ib X^{i+s} = \sum_{i\ge0} \eta^i r_i x X^i  \\
&= \sum_{i,j\ge0} \eta^i r_i x_j X^{i+j} \,.
\end{align*}
If $t>0$, it would follow that $\eta^0 r_0 x_0 = 0$, whence $r_0=0$. Then $r = r'X$ for some $r' \in R$, and so
$$
yX^t = r' Xx = \eta r' xX.
$$
But then $yX^{t-1} = \eta r' x \in Rx$, contradicting the minimality of $t$. Therefore $t=0$.

Consequently, $y = rx$, proving that $I = Rx$.

The proof that $I = xR$ is very similar, and is left to the reader.
\end{proof}

\subsection{$q$-skew calculations}
Since $\de \sig = \la_\circ^{-1} \sig \de$, the pair $(\sig,\de)$ is a \emph{$\la_\circ^{-1}$-skew derivation} in the terminology of \cite{qskew}. We shall need the following calculations.

The \emph{$q$-Leibniz Rules} for the $\la_\circ^{-1}$-skew situation \cite[Lemma 6.2]{qskew} say that
\begin{equation}  \label{qLeibniz}
\begin{aligned}
\delta^n(ef) &= \sum_{i=0}^n \binom{n}{i}_{\la_\circ^{-1}} \sig^{n-i} \de^i(e) \de^{n-i}(f)  \\
X^n e &= \sum_{i=0}^n \binom{n}{i}_{\la_\circ^{-1}} \sig^{n-i} \de^i(e) X^{n-i}
\end{aligned}
 \quad \forall\, n \in \Znn\,,\ e,f\in A,
\end{equation}
where the $q$-binomial coefficients, for $q= \la_\circ^{-1}$, are given by
\begin{equation}  \label{qnum}
\binom{n}{i}_q = \frac{(n)!_q}{(i)!_q (n-i)!_q} \,, \qquad (m)!_q = (m)_q (m-1)_q \cdots (1)_q \,, \qquad (m)_q = \frac{q^m-1}{q-1} \,.
\end{equation}

The argument of \cite[Lemme 7.2.3.2]{Ric} yields

\begin{lemma}  \label{Ric.arg}
Let $C$ be a $\KK$-algebra domain and $(\sig,\de)$ a $q$-skew derivation on $C$, where $q \in \kx$ is not a root of unity. Suppose $c,e \in C$ with $\de(c) = ce$ or $\de(c) = ec$. If there is some $m \in \Znn$ such that $\de^m(c) = \de^m(e) = 0$, then $\de(c) = 0$.
\end{lemma}

\begin{proof} We must show that one of $c$ or $e$ is zero. Suppose that $c,e \ne 0$, and let $s,t\in \Znn$ be maximal such that $\de^s(c), \de^t(e) \ne 0$. Assume first that $\de(c) = ce$. By the $q$-Leibniz Rule,
$$
\delta^{s+t}(ce) = \sum_{i=0}^{s+t} \binom{s+t}{i}_q \sig^{s+t-i} \de^i(c) \de^{s+t-i}(e) = \binom{s+t}{s}_q \sig^t\de^s(c) \de^t(e) \ne 0,
$$
since $\binom{s+t}{s}_q \ne 0$ because $q$ is not a root of unity. But then $\de^{s+t+1}(c) \ne 0$, due to the assumption $\de(c) = ce$. This is impossible, since $s+t+1 > s$. The assumption $\de(c) = ec$ leads to a similar contradiction.
\end{proof}


\sectionnew{Normal elements in Cauchon extensions}

Throughout this section, keep the assumptions of \S\S\ref{basicR}, \ref{cauchonext}, so that $R = A[X;\sig,\de]$ is a Cauchon extension.

\subsection{$\HH$-primes in Cauchon extensions}  \label{specstrat.bits}

By \cite[Lemmas 3.2, 3.3, Proposition 3.4 and their proofs]{GL},
\begin{enumerate}[(i)]
\item Every $\HH$-prime of $R$ is completely prime.
\item Every $\HH$-prime of $R$ contracts to a $\de$-stable $\HH$-prime of $A$.
\item For any $\de$-stable $\HH$-prime $P_0$ of $A$, there are at most two $\HH$-primes of $R$ that contract to $P_0$ in $A$. There is always at least one, namely $P_0R$.
\end{enumerate}
We shall also need the observation
\begin{enumerate}[(i)]
\item[(iv)] If $P$ is a prime (ideal) of $A$ (or $R$), then $(P:\HH) := \bigcap_{h\in \HH} (h\cdot P)$ is an $\HH$-prime of $A$ (or $R$).
\end{enumerate}
It follows that
\begin{enumerate}[(i)]
\item[(v)] If $I$ is an $\HH$-ideal of $A$ (or $R$), then all primes minimal over $I$ are $\HH$-primes.
\end{enumerate}

By the usual localization procedures for skew polynomial rings, $\sig$ and $\de$ extend uniquely to an automorphism and a $\sig$-derivation on $A^* := \Fract A$, and the skew polynomial algebra $R^* := A^*[X;\sig,\de]$ equals the localization of $R$ with respect to $A\setminus\{0\}$. The $\HH$-actions on $A$ and $R$ extend uniquely to actions on $A^*$ and $R^*$, and $(h_\circ\cdot) = \sig$ on $A^*$. Hence, except for local nilpotence of  $\de$, the assumptions of \S\S\ref{basicR}, \ref{cauchonext} also hold for $A^*$ and $R^*$.

Recall that an \emph{inner $\sig$-derivation} of $A^*$ (or $A$) is a map of the form $a \mapsto da - \sig(a)d$, for some fixed $d\in A^*$ (or $d\in A$). Such a $\sig$-derivation is denoted $\de_d$.

\begin{proposition}  \label{R*nonHsimple}
Assume that $R^*$ is not $\HH$-simple.

{\rm(a)} There is a unique element $d\in A^*$ such that $\de = \de_d$ on $A^*$ and $h\cdot d = \chi_X(h)d$ for all $h \in \HH$. In particular, $X-d$ is an $\HH$-eigenvector with $\chi_{X-d} = \chi_X$.

{\rm(b)} There is a unique nonzero $\HH$-prime in $R^*$, namely $(X-d)R^* = R^*(X-d)$.

{\rm(c)} Let $I^*$ be a proper nonzero $\HH$-ideal of $R^*$, let $n$ be the minimum degree for nonzero elements of $I^*$, and let $f = X^n + c X^{n-1} + [\text{\rm lower terms}]$, with $c\in A^*$, be a monic element of $I^*$ with degree $n$. Then $n>0$ and $d = (\la_\circ-1)(1-\la_\circ^n)^{-1} c$.
\end{proposition}

\begin{proof} These follow from \cite[Lemma 3.3]{GL} and its proof, since $A^*$ is $\HH$-simple. 
\end{proof} 

Whenever $R^*$ is not $\HH$-simple, we keep the notation $d$ for the element of $A^*$ described in Proposition \ref{R*nonHsimple}(a). Note that $R^* = A^*[X-d;\sig]$ in this case, and that items (i)--(v) above hold for $R^*$ and $A^*$.

\begin{corollary}  \label{uniquePcontract0}
If $R^*$ is not $\HH$-simple, then $(X-d) R^* \cap R$ is the unique nonzero $\HH$-prime of $R$ that contracts to $0$ in $A$. Moreover, any $\HH$-ideal of $R$ that contracts to $0$ in $A$ is contained in $(X-d) R^* \cap R$.
\end{corollary}

\begin{proof} On one hand, $P^* := (X-d)R^*$ is a nonzero $\HH$-prime of $R^*$ that contracts to $0$ in $A^*$, whence $P^* \cap R$ is a nonzero $\HH$-prime of $R$ that contracts to $0$ in $A$. On the other hand, any nonzero $\HH$-prime $Q$ of $R$ with $Q\cap A = 0$ localizes to a nonzero $\HH$-prime $Q R^*$ of $R^*$, whence $Q R^* = P^*$ and thus $Q = Q R^*\cap R = P^* \cap R$.

Similarly, any $\HH$-ideal $I$ of $R$ with $I\cap A = 0$ localizes to an $\HH$-ideal $I R^*$ of $R^*$. Since $I$ is disjoint from $A \setminus \{0\}$, we must have $IR^* \ne R^*$, whence there is at least one prime $Q^*$ of $R^*$ minimal over $IR^*$. Then $Q^*$ is an $\HH$-prime (\S\ref{specstrat.bits}(v)), whence $Q^* = P^*$. Therefore $I \subseteq IR^*\cap R \subseteq P^* \cap R$.
\end{proof}

\subsection{Some normal $\HH$-eigenvectors}

\begin{lemma}  \label{lcoeff=a}
Assume there is a nonzero $\HH$-prime $P$ in $R$ with $P\cap A = 0$. Let $a\in A$ be a normal $\HH$-eigenvector, and $s\in \Znn$ maximal such that $\de^s(a) \ne 0$.

{\rm(a)} If $s>0$, then $x := \theta(a) X^s$ is a normal $\HH$-eigenvector in $R$ and $x\in P$. Moreover, $d = \eta^{-1} (\la_\circ^s-1)^{-1} a^{-1} \de(a)$ and $\de(a) a = \eta \la_\circ^s a \de(a)$, where $\eta := \chi_a(h_\circ)$.

{\rm(b)} Now assume that $a$ is the leading coefficient of some element of $P$ with degree $1$. Then $a+P$ is normal in $R/P$. Moreover, if also $s=0$, then $\de \equiv 0$ and $P = XR$.
\end{lemma}

\begin{proof} The ideal $P$ localizes to a nonzero $\HH$-prime $P^*$ of $R^*$ such that $P^*\cap R = P$, and $P^* = (X-d)R^*$ by Proposition \ref{R*nonHsimple}(b). 

(a) By Lemma \ref{firstnormal}, $x$ is a normal $\HH$-eigenvector in $R$. 
Now $I := Rx$ is a nonzero $\HH$-ideal of $R$, and $I\cap A = 0$ because $\deg x = s>0$. By Corollary \ref{uniquePcontract0}, $I \subseteq P$, whence $x\in P$.

Note that $x = a X^s + c X^{s-1} + [\text{lower terms}]$, where 
$$
c = (1-\la_\circ)^{-1} \de\sig^{-1}(a) = \eta^{-1} (1-\la_\circ)^{-1} \de(a).
$$
The ideal $I$ localizes to a proper nonzero $\HH$-ideal $I^* := R^* x$ in $R^*$, and $s$ is the minimum degree for nonzero elements of $I^*$. Since $a^{-1} x$ is a monic element of $I^*$ with degree $s$, Proposition \ref{R*nonHsimple}(c) implies that 
$d = (\la_\circ-1) (1-\la_\circ^s)^{-1} a^{-1} c =  \eta^{-1} (\la_\circ^s-1)^{-1} a^{-1} \de(a)$. 

Observe that
$$
\de(a) = da - \eta ad =  \eta^{-1} (\la_\circ^s-1)^{-1} \bigl( a^{-1} \de(a) a - \eta a a^{-1} \de(a) \bigr),
$$
whence  $\eta (\la_\circ^s-1) \de(a) = a^{-1}\de(a)a - \eta \de(a)$, and therefore $\eta \la_\circ^s \de(a) = a^{-1}\de(a)a$.

(b) Assume that $aX + c \in P$ for some $c\in A$. Then $X + a^{-1}c$ is a monic element of $P^*$ with degree $1$. Since $P^*$ is proper, it contains no nonzero elements of degree $0$. Hence, we again apply Proposition \ref{R*nonHsimple}(c), obtaining $d = -a^{-1}c$.

If $s=0$, then $\de(a) = 0$, whence $\de^m(d) = - \eta^{-m} a^{-1} \de^m(c) = 0$ for some $m \in \Znn$. Since $\de(d) = dd - \sig(d)d = (1-\la_\circ)d^2$, it follows from Lemma \ref{Ric.arg} that $\de(d) = 0$. But $1-\la_\circ \ne 0$, so we obtain $d=0$. Thus $\de = \de_0 \equiv 0$ in this case. We then have $P = XR$. Moreover, $aX = \eta^{-1} Xa$, so $a$ is normal in $R$, whence also $a+P$ is normal in $R/P$.

Finally, assume that $s> 0$. By part (a), we have
$$
-a^{-1}c = d = \eta^{-1} (\la_\circ^s-1)^{-1} a^{-1} \de(a),
$$
whence $\de(a) = \eta (1-\la_\circ^s) c$. Since $aX+c \in P$, it follows that
$$
Xa = \eta aX + \eta (1-\la_\circ^s) c \equiv \eta aX + \eta (1-\la_\circ^s) (-aX) = \eta \la_\circ^s aX \quad \pmod{P}.
$$
As $a$ is already normal in $A$, we conclude that $a+P$ is normal in $R/P$.
\end{proof}

\begin{proposition} \label{somesep}
Assume that every nonzero $\HH$-prime of $A$ contains a normal $\HH$-eigen\-vec\-tor. 

If $P \subsetneq Q$ are $\HH$-primes of $R$ with $P\cap A = 0$, there exists a normal $\HH$-eigenvector $u$ of $R/P$ such that $u\in Q/P$.
\end{proposition}

\begin{proof}  Recall that $Q\cap A$ is a $\de$-stable $\HH$-prime of $A$.

Assume first that $P \ne 0$. Then $0$ and $P$ are two $\HH$-primes of $R$ that contract to $0$ in $A$, so $Q\cap A \ne 0$ by \S\ref{specstrat.bits}(iii).

Now $P$ localizes to a nonzero $\HH$-prime $P^*$ in $R^*$, and $P^* = R^*(X-d)$ by Proposition \ref{R*nonHsimple}(b). Writing $d = b^{-1}c$ for some $b,c\in A$ with $b\ne 0$, we have $bX-c = b(X-d) \in P^*\cap R = P$. Thus, the $\HH$-ideal 
$$J := \{ a \in A \mid aX+e \in P\ \text{for some}\ e \in A \}$$
is nonzero, as is then $J \cap (Q \cap A) = J\cap Q$. 

There exist primes $P_1,\dots,P_r$ in $A$ minimal over $J\cap Q$ such that $P_1P_2 \cdots P_r \subseteq J\cap Q$. Since $J\cap Q$ is an $\HH$-ideal, these $P_i$ are $\HH$-primes of $A$ (\S\ref{specstrat.bits}(v)). By hypothesis, each $P_i$ contains a normal $\HH$-eigenvector $a_i$, and thus $a := a_1a_2 \cdots a_r$ is a normal $\HH$-eigenvector of $A$ that lies in $J\cap Q$. Since $a$ is in $J$, it is the leading coefficient of an element of $P$ of degree $1$. By Lemma \ref{lcoeff=a}(b), the coset $u := a+P$ is a normal $\HH$-eigenvector of $R/P$. Moreover, $u \in Q/P$ because $a\in Q$.

Now assume that $P = 0$. If $Q\cap A \ne 0$, then by hypothesis, $Q\cap A$ contains a normal $\HH$-eigenvector $a$ of $A$. Then $\de^l(a) \in Q\cap A$ for all $l\in \Znn$, whence the element $u := \theta(a)X^s$ lies in $Q$, where $s\in \Znn$ is minimal such that $\theta(a) X^s \in R$. By Lemma \ref{firstnormal}, $u$ is a normal $\HH$-eigenvector in $R$.

Finally, suppose that $Q\cap A = 0$. As above, the $\HH$-ideal 
$$J := \{ a \in A \mid aX+e \in Q\ \text{for some}\ e \in A \}$$
is nonzero. If $J=A$, then $1\in J$, while if $J \ne A$, then $J$ contains a product of nonzero $\HH$-primes of $A$. In either case, there is a normal $\HH$-eigenvector $a$ of $A$ that lies in $J$. Let $s\in \Znn$ be maximal such that $\de^s(a) \ne 0$.

If $s>0$, then by Lemmas \ref{firstnormal} and \ref{lcoeff=a}(a), $u := \theta(a)X^s$ is a normal $\HH$-eigenvector of $R$ that lies in $Q$. On the other hand, if $s=0$, Lemma  \ref{lcoeff=a}(b) shows that $\de \equiv 0$ and $Q = XR$. In this case, $u := X$ is a normal $\HH$-eigenvector of $R$ that lies in $Q$. 
\end{proof}

\subsection{Carrying normal $\HH$-separation from $A$ to $R$}

\begin{definition} Suppose $C$ is a $\KK$-algebra equipped with an $\HH$-action by $\KK$-algebra automorphisms. Following \cite[\S5.2]{murcia.proc}, we say that $\Hspec C$ has \emph{normal $\HH$-separation} provided that for any proper inclusion $P \subsetneq Q$ of $\HH$-prime ideals of $C$, there exists a normal $\HH$-eigenvector of $C/P$ which lies in $Q/P$.

The condition of normal $\HH$-separation only requires a suitable supply of $\HH$-eigenvectors which are normal in appropriate factor rings. It does not require these normal elements to normalize via actions of elements of $\HH$. That requirement leads to the following stronger condition.  We say that $\Hspec C$ has \emph{$\HH$-normal separation} if, for any proper inclusion $P \subsetneq Q$ of $\HH$-prime ideals of $C$, the ideal $Q/P$ contains a nonzero element $u$ which is \emph{$\HH$-normal} in $C/P$, meaning that $u$ is normal and there is some $h\in \HH$ such that $uc = (h\cdot c)u$ for all $c\in C/P$.
\end{definition}

\begin{theorem}  \label{carryover.normalHsep}
If $\Hspec A$ has normal $\HH$-separation, then so does $\Hspec R$.
\end{theorem}

\begin{proof} Let $P \subsetneq Q$ be $\HH$-primes of $R$. Then $P_0 := P \cap A$ is a $\de$-stable $\HH$-prime of $A$ (\S\ref{specstrat.bits}(ii)), and we may replace $A$, $R$, $P$, $Q$ by $A/P_0$, $R/P_0R$, $P/P_0R$, $Q/P_0R$, respectively. Thus, there is no loss of generality in assuming that $P\cap A = 0$.

The hypothesis of normal $\HH$-separation now implies that every nonzero $\HH$-prime of $A$ contains a normal $\HH$-eigenvector of $A$. Therefore, by Proposition \ref{somesep}, there exists a normal $\HH$-eigen\-vector $u$ of $R/P$ such that $u\in Q/P$. This verifies normal $\HH$-separation in $\Hspec R$.
\end{proof}

\begin{question}  \label{Q1:Hnormalsep}
If $\Hspec A$ has $\HH$-normal separation, does $\Hspec R$ have $\HH$-normal separation?
\end{question}


\sectionnew{Proof of the main theorem}  \label{proofmaintheorem}

Observe that if $R =  \KK[x_1][x_2;\sigma_2,\delta_2]\cdots[x_N;\sigma_N,\delta_N]$ is a quantum nilpotent algebra, then $R_j$ is a Cauchon extension of $R_{j-1}$ for all $j \in \gc 2,N \dc$. (The complete primeness of $\HH$-primes follows from \S\ref{specstrat.bits}(i) by induction.) 

\begin{theorem}  \label{normsepCGL}
If $R$ is a quantum nilpotent algebra, then $\Spec R$ has normal separation.
\end{theorem}

\begin{proof} Write $R$ as in \eqref{itOre}, and let $\HH$ be as in Definition \ref{CGLdef}. Obviously $\Hspec R_0$ has normal $\HH$-separation. By induction on $N$, Theorem \ref{carryover.normalHsep} implies that $\Hspec R$ has normal $\HH$-separation. Therefore, by \cite[Theorem 5.3]{murcia.proc}, $\Spec R$ has normal separation.
\end{proof}

\begin{question}  \label{Q2:Hnormalsep}
If $R$ is a quantum nilpotent algebra, does $\Hspec R$ have $\HH$-normal separation?

As far as inclusions $0 \subsetneq Q$ of $\HH$-primes are concerned, $\HH$-normal separation is known to hold provided the torus $\HH$ is maximal in the sense of \cite[\S5.2]{GYncufd}. Namely, in this case all normal elements in $R$ itself are $\HH$-normal by \cite[Corollary 5.4]{GYncufd}.
\end{question}

We now address homological properties of a quantum nilpotent algebra $R$, some of which are obtained by filtering $R$ so that the associated graded ring $\gr R$ is a quantum affine space.

\begin{definition}  \label{qaffine}
A matrix $\bfq = (q_{ij}) \in M_N(\KK)$ is \emph{multiplicatively skew-symmetric} provided $q_{ii} = 1$ for all $i$ and $q_{ji} = q_{ij}^{-1}$ for all $i$, $j$. Given such a matrix, define the algebra
$$\OO_\bfq(\KK^N) := \KK \langle x_1,\dots,x_N \mid x_i x_j = q_{ij} x_j x_i \ \forall\; i,j \in \gc 1,N \dc \rangle.$$
The algebra $\OO_\bfq(\KK^N)$ is a quantized coordinate ring of the affine space $\Aset^N$, or a \emph{quantum affine space} for short. It is trivially a quantum nilpotent algebra.
\end{definition}

\begin{notation} \label{la-notation} If $R$ is a quantum nilpotent algebra as in Definition \ref{CGLdef}, there are scalars $\la_{ji} \in \kx$ such that $\sig_j(x_i) =  \la_{ji} x_i$ for $1\le i < j\le N$. These are the below-diagonal entries of a multiplicatively skew-symmetric matrix $\bfla = (\la_{ij}) \in M_N(\KK)$.
\end{notation}

\begin{lemma}  \label{filt}
Let $R$ be an iterated skew polynomial algebra of length $N$ as in \eqref{itOre}, and assume there is a multiplicatively skew-symmetric matrix $\bfq = (q_{ij}) \in M_N(\KK)$ such that $\sig_j(x_i) =  q_{ji} x_i$ for $1\le i < j\le N$. Then there exist an exhaustive, ascending, locally finite $\KK$-algebra filtration $(R_n)_{n\ge0}$ on $R$ and a $\KK$-al\-ge\-bra $\Znn$-grading on $\OO_{\bfq}(\KK^N)$ such that

{\rm(a)} $R_0 = \KK$.

{\rm(b)} The canonical generators $x_1,\dots,x_N$ of $\OO_{\bfq}(\KK^N)$ are homogeneous with positive degree.

{\rm(c)} $\gr R$ and $\OO_{\bfq}(\KK^N)$ are isomorphic as graded $\KK$-algebras, where the principal symbols of the $x_i$ in $R$ map to the $x_i$ in $\OO_{\bfq}(\KK^N)$.
\end{lemma}

\begin{proof} This is an application of \cite[Chapter 2, Corollary 3.3; Chapter 4, Proposition 6.4, Theorem 6.5]{BGTV}.
\end{proof}

\begin{proposition}  \label{ausregcm}
Let $R = \KK[x_1][x_2;\sigma_2,\delta_2]\cdots[x_N;\sigma_N,\delta_N]$ be an iterated skew polynomial algebra as in \eqref{itOre}, and assume that $\sig_j(x_i) \in \kx x_i$ for $1\le i<j\le N$. Then $R$ is an Auslander-regular, Cohen-Macaulay algebra of $GK$-dimension $N$.
\end{proposition}

\begin{proof} Auslander-regularity and the GK-dimension value follow by induction on $N$ from \cite[Theorem 4.2]{Eks} and \cite[Lemma 2.2]{HK}. Let $R$ be filtered as in Lemma \ref{filt}, so that $R_0 = \KK$ and $\gr R \cong \OO_{\bfq}(\KK^N)$. Then \cite[Theorem 3]{GTLob} implies that $R$ is Cohen-Macaulay.
\end{proof}

Now we have everything in hand to prove the main theorem.

\begin{proof}[First Proof of Theorem \ref{catCGL}] Clearly $R$ is an affine noetherian $\KK$-algebra domain. It is Auslander-Gorenstein and Cohen-Macaulay with finite GK-dimension by Proposition \ref{ausregcm}, and $\Spec R$ is normally separated by Theorem \ref{normsepCGL}. Therefore by \cite[Theorem 1.6]{GLen}, $\Spec R$ is catenary and Tauvel's height formula holds in $R$.

Now consider a prime ideal $P/Q$ in a prime quotient $R/Q$ of $R$. Due to catenarity in $\Spec R$, we have $\height(P/Q) = \height(P) - \height(Q)$. Taking account of the height formula for $R$, we obtain
\begin{align*}
\GK \bigl( (R/Q)/(P/Q) \bigr) + \height(P/Q) &= \GK(R/P) + \height(P) - \height(Q)  \\
&= \GK(R) - \height(Q) = \GK(R/Q),
\end{align*}
which verifies the height formula in $R/Q$.
\end{proof}

\begin{proof}[Second Proof of Theorem \ref{catCGL}] Catenarity follows from \cite[Theorem 0.1]{YZ}, whose hypotheses are verified as follows. (1) Normal separation is given by Theorem \ref{normsepCGL}. (2) If $R$ is filtered as in Lemma \ref{filt}, then $\gr R$ is graded isomorphic to $\OO_{\bfq}(\KK^N)$, which is clearly noetherian and connected graded. Moreover, $\OO_{\bfq}(\KK^N)$ has enough normal elements in the sense of \cite{YZ}, since if $P$ is a graded prime ideal of $\OO_{\bfq}(\KK^N)$ with $\OO_{\bfq}(\KK^N)/P \ne \KK$, then some $x_j \notin P$,  whence $x_j+P$ is a nonzero homogeneous normal element of $\OO_{\bfq}(\KK^N)/P$ with positive degree.

Tauvel's height formula for $R$ follows from \cite[Theorem 2.23]{YZ} or \cite[Theorem 7.1]{GLL}, and then the height formula may be established for prime quotients of $R$ as in the first proof.
\end{proof}



\begin{thebibliography}{99}

\bibitem{BGTV} J.L. Bueso, J. G\'omez-Torrecillas, and A. Verschoren, \emph{Algorithmic Methods in Non-Commutative Algebra. Applications to Quantum Groups}, Kluwer (2003) Dordrecht.

   \bibitem{Ca1} G. Cauchon, {\emph{Effacement des d\'erivations et spectres
premiers d'alg\`ebres quantiques,}} J. Algebra {\textbf{260}} (2003), 476--518.

\bibitem{Ca2} G. Cauchon, \emph{Spectre premier de $O_q(M_n(k))$. Image canonique et s\'eparation normale}, 
J. Algebra \textbf{260} (2003), 519--569.

\bibitem{Eks} E.K. Ekstr\"om, \emph{ The Auslander condition on graded and filtered Noetherian rings}, in
Seminaire Dubreil-Malliavin, 1987-1988, Lecture Notes in Mathematics, Vol. 1404 Berlin (1989) Springer, pp. 220--245.

\bibitem{Gab} O. Gabber, \emph{Equidimensionalit\'e de la vari\'et\'e
caract\'eristique}, Expos\'e de O. Gabber r\'edig\'e par T. Levasseur, Universit\'e de Paris VI, (1982).

\bibitem{GTLob} J. G\'omez-Torrecillas and F.J. Lobillo, \emph{Auslander-regular and Cohen-Macaulay quantum groups}, Algebras and Rep. Theory \textbf{7} (2004), 35--42.

\bibitem{qskew} K.R. Goodearl, \emph{Prime ideals in skew polynomial rings and quantized Weyl algebras}, J. Algebra \textbf{150} (1992), 324--377.

\bibitem{murcia.proc} K.R. Goodearl, \emph{Prime spectra of quantized coordinate
rings}, in Interactions between Ring Theory and
Representations of Algebras (Murcia 1998), (F.
Van Oystaeyen and M. Saor\'\i n, eds.), New York (2000), Dekker, pp. 205--237.

 \bibitem{gll1} K.R. Goodearl, S. Launois, and T.H. Lenagan, {\em Totally nonnegative cells and matrix Poisson varieties}, Advances in Math. \textbf{226} (2011), 779--826.

\bibitem{gll2} K.R. Goodearl, S. Launois, and T.H. Lenagan, {\em Torus-invariant prime ideals in quantum matrices, totally nonnegative cells and symplectic leaves},  Math. Zeitschrift \textbf{269} (2011), 29--45.

\bibitem{GLL} K.R. Goodearl, S. Launois, and T.H. Lenagan, \emph{Tauvel's height formula for 
quantum nilpotent algebras}, Communic. in Algebra \textbf{47} (2019), 4194--4209.

\bibitem{GLen} K.R. Goodearl and T.H. Lenagan, \emph{Catenarity in quantum algebras}, 
J. Pure Applied Algebra \textbf{111} (1996), 123--142.
   
\bibitem{GL} K.R. Goodearl and E.S. Letzter, \emph{The Dixmier--Moeglin 
equivalence in quantum coordinate rings and quantized Weyl algebras}, Trans. Amer. 
Math. Soc. {\textbf{352}} (2000), 1381--1403.  

\bibitem{GYncufd} K.R. Goodearl and M.T. Yakimov, \emph{From quantum Ore extensions to quantum tori via noncommutative UFDs}, Adv. Math. \textbf{300} (2015), 672--716. 

\bibitem{GZ} K.R. Goodearl and J.J. Zhang, \emph{Homological properties of quantized coordinate rings of semisimple groups}, Proc. Lond. Math. Soc. (3) \textbf{94} (2007), 647--671.

\bibitem{Hrt} K.L. Horton, \emph{The prime and primitive spectra of
multiparameter quantum symplectic and Euclidean spaces}, Communic. in Algebra \textbf{31} (2003), 2713--2743.

\bibitem{HK} C. Huh and C.O. Kim, \emph{Gelfand-Kirillov dimension of skew polynomial rings of automorphism type}, Communic. in Algebra \textbf{24} (1996), 2317--2323.

\bibitem{KrLen} G. Krause and T.H. Lenagan, \emph{Growth of Algebras and Gelfand-Kirillov Dimension}, Pitman (1985) Boston.

\bibitem{LLN} S. Launois, T.H. Lenagan, and B. Nolan, \emph{Total positivity is a quantum phenomenon: the grassmannian case}, arXiv:1906.06199.
  
   \bibitem{LLR} S. Launois, T.H. Lenagan, and L. Rigal, {\emph{Quantum unique factorisation domains}}, 
J. London Math. Soc. (2) {\textbf{74}} (2006), 321--340.

\bibitem{LLR2} S. Launois, T.H. Lenagan and L. Rigal, 
\emph{Prime ideals in the quantum grassmannian},
Selecta Math. (N.S.) \textbf{13} (2008), 697--725.

\bibitem{Len} T.H. Lenagan, \emph{Enveloping algebras of solvable Lie superalgebras are catenary}, Contemp. Math. \textbf{130} (1992), 231--236.

\bibitem{LevSt} T. Levasseur and J.T. Stafford, \emph{Rings of differential operators on classical rings of invariants}, Memoirs Amer. Math. Soc. No. 412 (1989).

\bibitem{Oh1} S.-Q. Oh, \emph{Catenarity in a class of iterated skew polynomial rings}, Communic. in Algebra \textbf{25} (1997), 37--49.

\bibitem{Ric} L. Richard, {\emph{Equivalence rationnelle et homologie de Hochschild pour certaines alg\`ebres polynomiales
classiques et quantiques}}, Th\'ese de doctorat, Universit\'e
Blaise Pascal (Clermont 2), (2002).

\bibitem{Sch} W. Schelter, \emph{Affine PI rings are catenary}, Bull. Amer. Math. Soc. \textbf{83} (1977), 1309--1310.

\bibitem{tauvel} P. Tauvel, \emph{Sur les quotients premiers de l'alg\`ebre enveloppante d'un alg\`ebre de Lie r\'esoluble}, Bull.
Soc. Math. France \textbf{106} (1978), 177--205.

\bibitem{Yak1} M. Yakimov, \emph{A proof of the Goodearl-Lenagan polynormality conjecture}, Int. Math. Res. Notices \textbf{9} (2013), 2097--2132.

\bibitem{Yak2} M. Yakimov, \emph{Spectra and catenarity of multi-parameter quantum Schubert cells}, Glasgow Math. J. \textbf{55A} (2013), 169--194.

\bibitem{Yak3} M. Yakimov, \emph{On the spectra of quantum groups}, Memoirs Amer. Math. Soc. No. 1078, \textbf{229} (2014).

\bibitem{YZ} A. Yekutieli and J.J. Zhang, \emph{Rings with Auslander dualizing complexes}, J. Algebra \textbf{213} (1999), 1--51.
   

\end{thebibliography}
\end{document}